\documentclass[11pt]{amsart}%
\usepackage{amsmath}%
\usepackage{amsfonts}%
\usepackage{amssymb}%
\usepackage{graphicx}
\newtheorem{theorem}{Theorem}[section]
\newtheorem{lemma}[theorem]{Lemma}
\newtheorem{proposition}[theorem]{Proposition}
\theoremstyle{definition}
\newtheorem{definition}[theorem]{Definition}
\theoremstyle{remark}
\newtheorem{remark}[theorem]{Remark}
\numberwithin{equation}{section}
\begin{document}
\title[Affine IFSs]{A characterization of hyperbolic affine iterated function systems}
\author{Ross Atkins}
\author{Michael F. Barnsley}
\address{Department of Mathematics\\
Australian National University\\
Canberra, ACT, Australia}
\email{michael.barnsley@maths.anu.edu.au, mbarnsley@aol.com}
\urladdr{http://www.superfractals.com{}}
\author{Andrew Vince}
\address{Department of Mathematics\\
University of Florida\\
Gainesville, FL 32611-8105, USA }
\email{avince@math.ufl.edu}
\urladdr{http://www.math.ufl.edu/$\sim$vince/}
\author{David C. Wilson}
\address{Department of Mathematics\\
University of Florida\\
Gainesville, FL 32611-8105, USA }
\email{dcw@math.ufl.edu}
\urladdr{http://www.math.ufl.edu/$\sim$dcw/}
\date{June 30, 2009}

\begin{abstract}
The two main theorems of this paper provide a characterization of hyperbolic
affine iterated function systems defined on $\mathbb{R}^{m}.$ Atsushi Kameyama
(\textit{Distances on Topological Self-Similar Sets,} Proceedings of Symposia
in Pure Mathematics, Volume \textbf{72.1,} 2004) asked the following
fundamental question: given a topological self-similar set, does there exist
an associated system of contraction mappings? Our theorems imply an
affirmative answer to Kameyama's question for self-similar sets derived from
affine transformations on $\mathbb{R}^{m}$.

\end{abstract}
\maketitle

\section{Introduction}

The goal of this paper is to prove and explain two theorems that classify
hyperbolic affine iterated function systems defined on $\mathbb{R}^{m}$. One
motivation was the question: when are the functions of an affine iterated
function systems (IFS) on $\mathbb{R}^{m}$ contractions with respect to a
metric equivalent to the usual euclidean metric?

\begin{theorem}
[Classification for Affine Hyperbolic IFSs]\label{maintheorem} If $\mathcal{F}
= \left(  \mathbb{R}^{m};f_{1} ,f_{2},...,f_{N}\right)  $ is an affine
iterated function system, then the following statements are equivalent.

\begin{enumerate}
\item $\mathcal{F}$ is hyperbolic.

\item $\mathcal{F}$ is point-fibred.

\item $\mathcal{F}$ has an attractor.

\item $\mathcal{F}$ is a topological contraction with respect to some convex
body $K \subset\mathbb{R}^{m}$.

\item $\mathcal{F}$ is non-antipodal with respect to some convex body $K
\subset\mathbb{R}^{m}$.
\end{enumerate}
\end{theorem}

Statement (1) is a metric condition on an affine IFS, statements (2) and (3)
are in terms of convergence, and statements (4) and (5) are in terms of
concepts from convex geometry. The terms contractive, hyperbolic,
point-fibred, attractor, topological contraction, and non-antipodal are
defined in Definitions \ref{contractiveIFSDef}, \ref{hyperbolicIFSDef},
\ref{point-fibredDef}, \ref{attractorIFSDef},
\ref{topologicalContractiveIFSDef}, \ref{nonantipodalIFSDef}, respectively.
This theorem draws together some of the main concepts in the theory of
iterated function systems. Banach's classical Contraction Mapping Theorem
states that a contraction $f$ on a complete metric space has a fixed point
$x_{0}$ and that $x_{0} = \lim_{k\rightarrow\infty} f^{\circ k} (x)$,
independent of $x$, where $\circ k$ denotes the $k^{th}$ iteration. The notion
of hyperbolic generalizes to the case of an IFS the contraction property,
namely an IFS is hyperbolic if there is a metric on $\mathbb{R}^{m},$
equivalent to the usual one, such that each function in the IFS is a
contraction. The notion of point-fibred, introduced by Kieninger
\cite{kieninger}, is the natural generalization of the limit condition above
to the case of an IFS. While traditional discussions of fractal geometry focus
on the existence of an attractor for a hyperbolic IFS,
Theorem~\ref{maintheorem} establishes that the more geometrical (and
non-metric) assumptions - topologically contractive and non-antipodal - can
also be used to guarantee the existence of an attractor. Basically a function
$f\,: \, \mathbb{R}^{m} \rightarrow\mathbb{R}^{m}$ is non-antipodal if certain
pairs of points (antipodal points) on the boundary of $K$ are not mapped by
$f$ to another pair of antipodal points.

Since the implication $(1) \Rightarrow(2)$ is the Contraction Mapping Theorem
when the IFS contains only one affine mapping, Theorem~\ref{maintheorem}
contains an affine IFS version of the converse to the Contraction Mapping
Theorem. Thus, our theorem provides a generalization of results proved by L.
Janos \cite{janos} and S. Leader \cite{leader}. Such a converse statement in
the IFS setting has remained unclear until now.

Although not every affine IFS $\mathcal{F} = \left(  \mathbb{R}^{m};f_{1}
,f_{2},...,f_{N}\right)  $ is hyperbolic on all of $\mathbb{R}^{m}$, the
second main result states that if $\mathcal{F}$ has a coding map
(Definition~\ref{codingMapDef}), then $\mathcal{F}$ is always hyperbolic on
some affine subspace of $\mathbb{R}^{m}$.

\begin{theorem}
\label{kameyamaTheorem} If $\mathcal{F}= \left(  \mathbb{R}^{m}; f_{1}
,f_{2},...,f_{N}\right)  $ is an affine IFS with a coding map $\pi:
\Omega\rightarrow\mathbb{R}^{m}$, then $\mathcal{F}$ is hyperbolic on the
affine hull of $\pi(\Omega)$. In particular, if $\pi(\Omega)$ contains a
non-empty open subset of $\mathbb{R}^{m}$, then $\mathcal{F}$ is hyperbolic on
$\mathbb{R}^{m}$.
\end{theorem}

Although he used slightly different terminology, Kameyama \cite{kameyama}
posed the following \textit{FUNDAMENTAL\ QUESTION}: \textit{Is an affine IFS
with a coding map }$\pi:\Omega\rightarrow\mathbb{R}^{m}$\textit{ hyperbolic
when restricted to }$\pi(\Omega)$\textit{?} An affirmative answer to this
question follows immediately from Theorem \ref{kameyamaTheorem}.

Our original motivation, however, was not Kameyama's question, but rather a
desire to approximate a compact subset $T\subset\mathbb{R}^{m}$ as the
attractor $A$ of an iterated function system $\mathcal{F}=(\mathbb{R}^{m};
f_{1},f_{2},...f_{N}),$ where each $f_{n}:\mathbb{R}^{m}\rightarrow
\mathbb{R}^{m}$ is affine. This task is usually done using the ``collage
theorem" \cite{barnsleynotices}, \cite{barnsleyPNAS} by choosing an IFS
$\mathcal{F}$ so that the Hausdorff distance $d_{\mathbb{H} }(T,\mathcal{F}%
\left(  T\right)  )$ is small. If the IFS $\mathcal{F}$ is hyperbolic, then we
can guarantee it has an attractor $A$ such that $d_{\mathbb{H}}(T,A)$ is
comparably small. But then the question arises: how does one know if
$\mathcal{F}$ is hyperbolic?

The paper is organized as follows. Section 2 contains notation, terminology,
and definitions that will be used throughout the paper. Section 3 contains
examples and remarks relating iterated function systems and their attractors
to Theorems \ref{maintheorem} and \ref{kameyamaTheorem}. In Example 3.1 we
show that an affine IFS can be point-fibred, but not contractive under the
usual metric on $\mathbb{R}^{m}$. Thus, some kind of remetrization is required
for the system to be contractive. In Example 3.2 we show that an affine IFS
can contain two linear maps each with real eigenvalues all with magnitudes
less than 1, but still may not be point-fibered. Thus, Theorem
\ref{maintheorem} cannot be phrased only in terms of eigenvalues and
eigenvectors of the individual functions in the IFS. Indeed, in Example 3.3 we
explain how, given any integer $M>0$, there exists a linear IFS $\left(
\mathbb{R}^{2};L_{1},L_{2}\right)  $ such that each operator of the form
$L_{\sigma_{1}}L_{\sigma_{2}}...L_{\sigma_{k}}$, with $\sigma_{j}\in\{1,2\}$
for $j=1,2,...,k$, and $k\leq M,$ has spectal radius less than one, while
$L_{1}L_{2}^{M}$ has spectral radius larger than one. This is related to the
joint spectral radius \cite{strang} of the pair of linear operators
$L_{1},L_{2}$ and to the associated finiteness conjecture, see for example
\cite{theys}. In Section 8 we comment on the relationship between the present
work and recent results concerning the joint spectral radius of finite sets of
linear operators. Example 3.4 provides an affine IFS on $\mathbb{R}^{2}$ that
has a coding map $\pi$, but is not point-fibred on $\mathbb{R}^{2}$, and hence
by Theorem~\ref{maintheorem}, not hyperbolic on $\mathbb{R}^{2}$. It is,
however, point-fibred and hyperbolic when restricted to the $x$-axis, which is
the affine hull of $\pi(\Omega)$, thus illustrating
Theorem~\ref{kameyamaTheorem}.

For the proof of Theorem \ref{maintheorem} we provide the following roadmap.

\begin{enumerate}
\item The proof that statement (1) $\Rightarrow$ statement (2) is provided in
Theorem \ref{1Implies2Theorem}.

\item The proof that statement (2) $\Rightarrow$ statement (3) is provided in
Theorem \ref{2Implies3Theorem}.

\item The proof that statement (3) $\Rightarrow$ statement (4) is provided in
Theorem \ref{3Implies4Theorem}.

\item The proof that statement (4) $\Rightarrow$ statement (5) is provided in
Proposition \ref{topContractImpliesnonAntpodalProp}.

\item The proof that statement (5) $\Rightarrow$ statement (1) is provided in
Theorem~\ref{5Implies1Theorem}.
\end{enumerate}

Theorem~\ref{kameyamaTheorem} is proved in section 7.

\section{Notation and Definitions}

We treat $\mathbb{R}^{m}$ as a vector space, an affine space, and a metric
space. We identify a point $x = (x_{1},x_{2},...,x_{m})\in\mathbb{R}^{m}$ with
the vector whose coordinates are $x_{1},x_{2},...,x_{m}$. We write
$0\in\mathbb{R}^{m}$ for the point in $\mathbb{R}^{m}$ whose coordinates are
all zero. The standard basis is denoted $\{e_{1}, e_{2}, \dots, e_{m}\}$. The
inner product between $x,y\in\mathbb{R}^{m}$ is denoted by $\langle
x,y\rangle$. The $2$-norm of a point $x\in\mathbb{R}^{m}$ is $\left\Vert
x\right\Vert _{2} = \sqrt{\langle x,x\rangle}$, and the euclidean metric
$d_{E}:\mathbb{R}^{m}\times\mathbb{R}^{m}\rightarrow\lbrack0,\infty)$ is
defined by $d_{E}(x,y)= \left\Vert x-y\right\Vert _{2} \text{ for all }%
x,y\in\mathbb{R}^{m}$. The following notations, conventions, and definitions
will also be used throughout this paper:

\begin{enumerate}
\item A convex body is a compact convex subset of $\mathbb{R}^{m}$ with
non-empty interior.

\item For a set $B$ in $\mathbb{R}^{m}$, the notation $conv(B)$ is used to
denote the convex hull of $B$.

\item For a set $B \in\mathbb{R}^{m}$, the \textit{affine hull}, denoted
$\text{aff}( B )$, of $B$ is the smallest affine subspace containing $B$,
i.e., the intersection of all affine subspaces containing $B$.

\item The symbol $\mathbb{H}$ will denote the nonempty compact subsets of
$\mathbb{R}^{m}$, and the symbol $d_{\mathbb{H}}$ will denote the Hausdorff
metric on $\mathbb{H}$. Recall that $(\mathbb{R}^{m} ,d_{\mathbb{H}})$ is a
complete metric space.

\item A metric $d$ on $\mathbb{R}^{m}$ is said to be \textit{Lipschitz
equivalent} to $d_{E}$ if there are positive constants $r$ and $R$ such that
\[
r \, d_{E}(x,y) \leq d(x,y) \leq R \, d_{E}(x,y),
\]
for all $x,y \in\mathbb{R}^{m}$. If two metrics are Lipschitz equivalent, then
they induce the same topology on $\mathbb{R}^{m}$, but the converse is not
necessarily true.

\item For any two subsets $A$ and $B$ of $\mathbb{R}^{m}$ the notation $A - B
:= \{x -y \, : x \in A \text{ and }y \in B\}$ is used to denote the pointwise
subtraction of elements in the two sets.

\item For a positive integer $N,$ the symbol $\Omega= \{1,2, \dots
,N\}^{\infty} $ will denote the set of all infinite sequences of symbols
$\{\sigma_{k}\}_{k=1}^{\infty}$ belonging to the alphabet $\{1,2, \dots,N\}$.
The set $\Omega$ is endowed with the product topology. An element of
$\sigma\in\Omega$ will also be denoted by the concatenation $\sigma=
\sigma_{1}\sigma_{2}\sigma_{3} \dots$, where $\sigma_{k}$ denotes the $k^{th}$
component of $\sigma$. Recall that since $\Omega$ is endowed with the product
topology, it is a compact Hausdorff space.
\end{enumerate}

\begin{definition}
[IFS]If $N>0$ is an integer and $f_{n} :\mathbb{R}^{m} \rightarrow
\mathbb{R}^{m}$, $n=1,2, \dots, N,$ are continuous mappings, then
${\mathcal{F}} = \left(  \mathbb{R}^{m}; f_{1} ,f_{2},...,f_{N}\right)  $ is
called an \textit{iterated function system} (IFS). If each of the functions in
${\mathcal{F}}$ is an affine map on $\mathbb{R}^{m}$, then ${\mathcal{F}}$ is
called an \textit{affine IFS}.
\end{definition}

\begin{definition}
[Contractive IFS]\label{contractiveIFSDef} An IFS $\mathcal{F} = \left(
\mathbb{R}^{m}; f_{1} ,f_{2},...,f_{N}\right)  $ is \textit{contractive} when
each $f_{n}$ is a contraction. Namely, there is a number $\alpha_{n}\in
\lbrack0,1)$ such that $d_{E}(f_{n}(x),f_{n}(y)) \leq\alpha_{n}d_{E}(x,y)$ for
all $x,y\in\mathbb{R}^{m}$, for all $n$.
\end{definition}

\begin{definition}
[Hyperbolic IFS]\label{hyperbolicIFSDef} An IFS $\mathcal{F} = \left(
\mathbb{R}^{m}; f_{1} ,f_{2},...,f_{N}\right)  $ is called \textit{hyperbolic}
if there is a metric on $\mathbb{R}^{m}$ Lipschitz equivalent to the given
metric so that each $f_{n}$ is a contraction.
\end{definition}

\begin{definition}
[Coding Map]\label{codingMapDef} A continuous map $\pi:\Omega\rightarrow
\mathbb{R}^{m}$ is called a \textit{coding map} for the IFS $\mathcal{F} =
\left(  \mathbb{R}^{m}; f_{1} ,f_{2},...,f_{N}\right)  $ if, for each
$n=1,2,\dots, N,$ the following diagram commutes,
\begin{equation}%
\begin{array}
[c]{ccc}%
\Omega & \overset{s_{n}}{\rightarrow} & \Omega\\
\pi\downarrow\text{\ \ \ \ } &  & \text{ \ \ \ }\downarrow\pi\\
\mathbb{R}^{m} & \underset{f_{n}}{\rightarrow} & \mathbb{R}^{m}%
\end{array}
\label{commutediagram}%
\end{equation}
where the symbol $s_{n}:\Omega\rightarrow\Omega$ denotes the inverse shift map
defined by $s_{n}(\sigma)=n \sigma$.
\end{definition}

The notion of a coding map is due to J. Kigami \cite{kigami} and A. Kameyama
\cite{kameyama}.

\begin{definition}
[Point-Fibred IFS]\label{point-fibredDef} An IFS $\mathcal{F} = \left(
\mathbb{R}^{m}; f_{1} ,f_{2},...,f_{N}\right)  $ is \textit{point-fibred} if,
for each $\sigma= \sigma_{1}\,\sigma_{2} \, \sigma_{3} \cdots\in\Omega,$ the
limit on the right hand side of
\begin{equation}
\label{limitdef}\pi(\sigma) := \lim_{k\rightarrow\infty}f_{\sigma_{1}}\circ
f_{\sigma_{2}} \circ\cdots\circ f_{\sigma_{k}}(x),
\end{equation}
exists, is independent of $x \in\mathbb{R}^{m}$ for fixed $\sigma$, and the
map $\pi:\Omega\rightarrow\mathbb{R}^{m}$ is a coding map.
\end{definition}

It is not difficult to show that \ref{limitdef} is the unique coding map of a
point-fibred IFS. Our notion of a point-fibred iterated function system is
similar to Kieninger's Definition 4.3.6 \cite{kieninger}, p.97. However, we
work in the setting of complete metric spaces whereas Kieninger frames his
definition in a compact Hausdorff space.

\begin{definition}
[The Symbol $\mathcal{F}(B)$ for an IFS]\label{decompositionIFSDef} For an IFS
$\mathcal{F} = \left(  \mathbb{R}^{m}; f_{1} ,f_{2},...,f_{N}\right)  $ define
$\mathcal{F} :\mathbb{H\rightarrow}\mathbb{H}$ by
\[
\mathcal{F}(B)= \bigcup_{n=1}^{N} f_{n}(B).
\]
(We use the same symbol $\mathcal{F}$ both for the IFS and the mapping.) For
$B \in\mathbb{H}$, let $\mathcal{F}^{\circ k}(B)$ denote the $k$-fold
composition of $\mathcal{F}$, i.e., the union of $f_{\sigma_{1}}\circ
f_{\sigma_{2}}\circ\cdots\circ f_{\sigma_{k}} (B)$ over all words $\sigma
_{1}\sigma_{2}\cdots\sigma_{k}$ of length $k$.
\end{definition}

\begin{definition}
[Attractor for an IFS]\label{attractorIFSDef} A set $A \in\mathbb{H}$ is
called an \textit{attractor} of an IFS $\mathcal{F} = \left(  \mathbb{R}^{m};
f_{1} ,f_{2},...,f_{N}\right)  $ if
\begin{equation}
A = \mathcal{F}(A) \label{selfrefeq}%
\end{equation}
and
\begin{equation}
A = \lim_{k\rightarrow\infty} \mathcal{F}^{\circ k}(B), \label{convergeq}%
\end{equation}
the limit with respect to the Hausdorff metric, for all $B\in\mathbb{H}$.
\end{definition}

If an IFS has an attractor $A$, then clearly $A$ is the unique attractor. It
is well known that a hyperbolic IFS has an attractor. An elegant proof of this
fact is given by J. Hutchinson \cite{hutchinson}. He observes that a
contractive IFS $\mathcal{F}$ induces a contraction $\mathcal{F}%
:\mathbb{H\rightarrow H}$, from which the result follows by the contraction
mapping theorem. See also M. Hata \cite{hata} and R. F. Williams
\cite{williams}.

In section 4 it is shown that a point-fibred IFS $\mathcal{F}$ has an
attractor $A$, and, moreover, if $\pi$ is the coding map of $\mathcal{F}$,
then $A = \pi(\Omega)$. Often $\sigma$ is considered as the ``address'' of the
point $\pi(\sigma)$ in the attractor. In the literature on fractals (for
example J. Kigami \cite{kigami}) there is an approach to the concept of a
self-similar system without reference to the ambient space. This approach
begins with the idea of a continuous coding map $\pi$ and, in effect, defines
the attractor as $\pi(\Omega)$.

\section{Examples and Remarks on Iterated Function Systems}

This section contains examples and remarks relevant to
Theorems~\ref{maintheorem} and \ref{kameyamaTheorem}.\bigskip

\noindent\textbf{EXAMPLE 3.1} [A Point-fibred, not Contractive IFS] Consider
the affine IFS consisting of a single linear function on $\mathbb{R}^{2}$
given by the matrix
\[
f=%
\begin{pmatrix}
0 & 2\\
\frac{1}{8} & 0
\end{pmatrix}
.
\]
Note that the eigenvalues of $f$ equal $\pm\frac{1}{2}$. Since
\[
\lim_{n\rightarrow\infty}f^{\circ2n}=\lim_{n\rightarrow\infty}T^{-1}%
\begin{pmatrix}
(\frac{1}{2})^{n} & 0\\
0 & (-\frac{1}{2})^{n}%
\end{pmatrix}
T=%
\begin{pmatrix}
0 & 0\\
0 & 0
\end{pmatrix}
,
\]
where $T$ is the change of basis matrix, this IFS is point-fibred. However,
since
\[
f%
\begin{pmatrix}
0\\
1
\end{pmatrix}
=%
\begin{pmatrix}
2\\
0
\end{pmatrix}
,
\]
the mapping is not a contraction under the usual metric on $\mathbb{R}^{2}$.
Theorem \ref{maintheorem}, however, guarantees we can remetrize $\mathbb{R}%
^{2}$ with an equivalent metric so that $f$ is a contraction.\bigskip

\noindent\textbf{EXAMPLE 3.2} [An IFS with Point-Fibred Functions that is not
Point-Fibred] In the literature on affine iterated function systems, it is
sometimes assumed that the eigenvalues of the linear parts of the affine
functions are less than $1$ in modulus. Unfortunately, this assumption is not
sufficient to imply any of the five statements given in
Theorem~\ref{maintheorem}. While the affine IFS $(\mathbb{R}^{m};f)$ is
point-fibred if and only if the eigenvalues of the linear part of $f$ all have
moduli strictly less than $1$, an analogous statement cannot be made if the
number of functions in the IFS is larger than $1$.

Consider the affine IFS $\mathcal{F} = \left(  \mathbb{R}^{2}; f_{1} ,f_{2}
\right)  ,$ where
\[
f_{1} =
\begin{pmatrix}
0 & 2\\
\frac18 & 0
\end{pmatrix}
\qquad\hbox{and} \qquad f_{2} =
\begin{pmatrix}
0 & \frac18\\
2 & 0
\end{pmatrix}
.
\]

\noindent As noted in Example 3.1%
\[
\lim_{n\rightarrow\infty}f_{1}^{\circ n}\mathbf{u}=\lim_{n\rightarrow\infty
}f_{2}^{\circ n}\mathbf{u}=%
\begin{pmatrix}
0\\
0
\end{pmatrix}
\]
for any vector $\mathbf{u}$. Thus, both $\mathcal{F}_{1}=\left(
\mathbb{R}^{2};f_{1}\right)  $ and $\mathcal{F}_{2}=\left(  \mathbb{R}%
^{2};f_{2}\right)  $ are point-fibred. Unfortunately, their product is the
matrix
\[
f_{1}\circ f_{2}=%
\begin{pmatrix}
4 & 0\\
0 & \frac{1}{64}%
\end{pmatrix}
,\quad\text{so that}\quad\lim_{n\rightarrow\infty}(f_{1}\circ f_{2})^{\circ n}%
\begin{pmatrix}
1\\
0
\end{pmatrix}
=\lim_{n\rightarrow\infty}%
\begin{pmatrix}
4^{n}\\
0
\end{pmatrix}
=+\infty.
\]
Thus, the IFS $\mathcal{F}=\left(  \mathbb{R}^{2};f_{1},f_{2}\right)  $ fails
to be point-fibred.

\begin{remark}
\label{converseRemark}
\end{remark}

While it is true that $(1)\Rightarrow(2)$ in Theorem~\ref{maintheorem} even
without the assumption that the IFS is affine, the converse is not true in
general. Kameyama \cite{kameyama} has shown that there exists a point-fibred
IFS that is not hyperbolic. We next give an example of an affine IFS with a
coding map that is not point-fibred. Thus, the set of IFSs (with a coding map)
strictly contains the set of point-fibred IFSs which, in turn, strictly
contains the set of hyperbolic IFSs.\bigskip

\noindent\textbf{EXAMPLE 3.3 [}The Failure of a Finite Eigenvalue Test to
Imply Point-Fibred] Consider the linear IFS $\mathcal{F}=\left(
\mathbb{R}^{2};L_{1},L_{2}\right)  ,$ where
\[
L_{1}=%
\begin{pmatrix}
0 & 2\\
\frac{1}{8} & 0
\end{pmatrix}
\qquad\hbox{and}\qquad L_{2}=%
\begin{pmatrix}
a\cos\theta & -a\sin\theta\\
a\sin\theta & a\cos\theta
\end{pmatrix}
=aR_{\theta},
\]
where $R_{\theta}$ denotes rotation by angle $\theta$, and $0<$ $a<1.$ Then
$L_{1}^{n}$ has eigenvalues $\pm1/2^{n}$ while the eigenvalues of $L_{2}^{n}$
both have magnitude $a^{n}<1$. For example, if we choose $\theta=\pi/8$ and
$a=31/32$ then it is readily verified that the eigenvalues of $L_{1}L_{2}$ and
$L_{2}L_{1}$ are smaller than one in magnitude and that one of the eigenvalues
of $L_{1}L_{2}L_{2}$ is $1.4014...$ . Hence, in this case, the magnitudes of
the eigenvalues of the linear operators $L_{1},$ $L_{2},$ $L_{1}^{2},$
$L_{1}L_{2},$ $L_{2}L_{1},$ $L_{2}^{2}$ are all less than one, but $\left\Vert
\left(  L_{1}L_{2}L_{2}\right)  ^{n}x\right\Vert $ does not converge when
$x\in\mathbb{R}^{2}$ is any eigenvector of $L_{1}L_{2}L_{2}$ corresponding to
the eigenvalue $1.4014...$. It follows that the IFS $\left(  \mathbb{R}%
^{2};L_{1},L_{2}\right)  $ is not point-fibred. By using the same underlying
idea it is straightforward to prove that, given any positive integer $M$, we
can choose $a$ close to $1$ and $\theta$ close to $0$ in such a way that the
eigenvalues of $L_{\sigma_{1}}L_{\sigma_{2}}...L_{\sigma_{k}},$ (where
$\sigma_{j}\in\{1,2\}$ for $j=1,2,...,k$, with $k\leq M$) are all of magnitude
less than one, while $L_{1}L_{2}^{M}$ has an eigenvalue of magnitude larger
than one.\bigskip

\noindent\textbf{EXAMPLE 3.4} [A non-Hyperbolic Affine IFS] Let $\mathcal{F}%
=\left(  \mathbb{R}^{2};f_{0},f_{1}\right)  ,$ where
\[
f_{0}(x_{1},x_{2})=(\frac{1}{2}\,x_{1},x_{2}),\qquad\qquad f_{1}(x_{1}%
,\,x_{2})=(\frac{1}{2}\,x_{1}+\frac{1}{2},\,x_{2}).
\]

\noindent This IFS has a coding map $\pi$ with $\Omega=\{0,1\}^{\infty}$ and
$\pi(\sigma)=(0.\sigma,0),$ where $0.\sigma$ is considered as a base 2
decimal. Since $\lim_{k\rightarrow\infty}f_{\sigma_{1}}\circ f_{\sigma_{2}%
}\circ\cdots\circ f_{\sigma_{k}}(x_{1},x_{2})=(0.\sigma,x_{2})$ depends on the
choice of the points $(x_{1},x_{2})\in\mathbb{R}^{2},$ this IFS cannot be
point-fibred. Hence, by Theorem~\ref{maintheorem}, the IFS $\mathcal{F}$ is
also not hyperbolic. However, it is clearly hyperbolic when restricted to the
$x$-axis, the affine hull of unit interval $\pi(\Omega)=[0,1]\times\{0\}$.
Thus, this example illustrates Theorem~\ref{kameyamaTheorem}.

A key fact used in the proof of Theorem \ref{maintheorem} is that the set of
antipodal points in a convex body equals the set of diametric points. The
definitions of antipodal and diametric points are given in Definitions
\ref{antipodalPointsDef} and \ref{diametricPointsDef}, respectively. The
equality between these two point sets is proved in Theorem \ref{keyADTheorem}.
While it is possible that this result is present in the convex geometry
literature, it does not seem to be well-known. For example, it is not
mentioned in the works of Moszynska \cite{moszynska} or Schneider
\cite{schneider}. This equivalence between antipodal and diametric points is
crucial to our work because it provides the remetrization technique at the
heart of Theorem \ref{5Implies1Theorem}, which implies that a non-antipodal
IFS is hyperbolic. A consequence of Theorem~\ref{maintheorem} is that a
non-antipodal affine IFS has the seemingly stronger property of being
topologically contractive.

\section{Hyperbolic Implies Point-fibred Implies The Existence of an
Attractor}

The implications $(1) \Rightarrow(2)\Rightarrow(3)$ in
Theorem~\ref{maintheorem} are proved in this section. For this section we also
introduce the notation $f_{\sigma\,|\,k} = f_{\sigma_{1}}\circ f_{\sigma_{2}}
\circ\cdots\circ f_{\sigma_{k}}(x)$. Note that, for $k$ fixed, $f_{\sigma
\,|\,k}(x)$ is a function of both $x$ and $\sigma$.

\begin{theorem}
\label{1Implies2Theorem} If $\mathcal{F}=\left(  \mathbb{R}^{m}; f_{1}
,f_{2},...,f_{N}\right)  $ is a hyperbolic IFS, then $\mathcal{F}$ is point fibred.
\end{theorem}

\begin{proof}
For $\sigma\in\Omega$, the proof that the limit $\lim_{k\rightarrow\infty
}f_{\sigma| k}$ exists and is independent of $x$ is virtually identical to the
proof of the classical Contraction Mapping Theorem. Moreover, the same proof
shows that the limit is uniform in $\sigma$.

With $\pi:\Omega\rightarrow\mathbb{R}^{m}$ defined by $\pi(\sigma) =
\lim_{k\rightarrow\infty}f_{\sigma| k}$ it is easy to check that, for each
$n=1,2,\dots, N$, the diagram \ref{commutediagram} commutes.

It only remains to show that $\pi$ is continuous. With $x$ fixed,
$f_{\sigma\,|\,k}(x)$ is a continuous function of $\sigma$. This is simply
because, if $\sigma, \tau\in\Omega$ are sufficiently close in the product
topology, then they agree on the first $k$ components. By
Definition~\ref{point-fibredDef}, the function $\pi$ is then the uniform limit
of continuous (in $\sigma$) functions defined on the compact set $\Omega$.
Therefore, $\pi$ is continuous.
\end{proof}

Let $\mathcal{F}$ be a point-fibred affine IFS, and let $A$ denote the set
\[
A := \pi(\Omega).
\]
According to Theorem~\ref{2Implies3Theorem}, $A$ is the attractor of
$\mathcal{F}$.

\begin{lemma}
\label{uniformLemma} Let $\mathcal{F}=\left(  \mathbb{R}^{m}; f_{1}
,f_{2},...,f_{N}\right)  $ be a point-fibred affine IFS with coding map
$\pi:\Omega\rightarrow\mathbb{R}^{m}$. If $B \subset\mathbb{R}^{m}$ is
compact, then the convergence in the limit
\[
\pi(\sigma) = \lim_{k\rightarrow\infty} f_{\sigma|k}(x)
\]
is uniform in $\sigma= \sigma_{1} \sigma_{2} \cdots\in\Omega$ and $x\in B$ simultaneously.
\end{lemma}

\begin{proof}
Only the uniformity requiress proof. Express $f_{n}(x) = L_{n}x + a_{n}$,
where $L_{n}$ is the linear part. Then
\begin{equation}
\label{linear}\begin{aligned} f_{\sigma|k}(x) &= L_{\sigma|k}(x) + L_{\sigma|k-1}(a_{\sigma_k})+ L_{\sigma|k-2}(a_{\sigma_{k-1}}) + \cdots + L_{\sigma|1}a_2 + a_1 \\ &= L_{\sigma|k}(x) + f_{\sigma|k}(0). \end{aligned}
\end{equation}
From equation~\ref{linear} it follows that, for any $x,y\in B$,
\begin{equation}
\label{close1}\begin{aligned} d_E(f_{\sigma|k}(x),f_{\sigma|k}(y)) &= \left\Vert L_{\sigma|k}(x-y) \right \Vert_2 \\ &\leq \sup\, \Bigl\{ \sum_{j=1}^m 2\,|c_j| \left\Vert L_{\sigma|k}(e_j) \right\Vert_2 \, : \, c_1 e_1 + \cdots + c_m e_m \in B \Bigr\} \\ &\leq c\, \max_j \, \left\Vert f_{\sigma|k}(e_j) -f_{\sigma|k}(0) \right\Vert_2, \end{aligned}
\end{equation}
where $c = 2 m \cdot\sup\, \{ \max_{j} |c_{j}|\, : c_{1} e_{1} + \cdots+ c_{m}
e_{m} \in B\}$ and where $\{e_{j}\}_{j=1}^{m}$ is a basis for $\mathbb{R}^{m}$.

Let $\epsilon>0$. From the definition of point-fibred there is a $k_{j}$,
independent of $\sigma$, such that if $k>k_{j}$, then
\[
\left\Vert f_{\sigma|k}(e_{j}) - \pi(\sigma) \right\Vert _{2} < \frac
{\epsilon}{4c} \qquad\text{and} \qquad\left\Vert f_{\sigma|k}(0) - \pi(\sigma)
\right\Vert _{2} < \frac{\epsilon}{4c},
\]
which implies $\left\Vert f_{\sigma|k}(e_{j}) - f_{\sigma|k}(0) \right\Vert
_{2} < \frac{\epsilon}{2c}$. This and equation~\ref{close1} implies that if
$k\geq{{\overline k}} := \max_{j} k_{j}$, then for any $x,y\in B$ we have
\begin{equation}
\label{close2}d_{E}(f_{\sigma|k}(x),f_{\sigma|k}(y)) < c\frac{\epsilon}{2c} =
\frac{\epsilon}{2}.
\end{equation}
Let $b$ be a fixed element of $B$. There is a $k_{b}$, independent of $\sigma
$, such that if $k>k_{b}$, then $d_{E}(f_{\sigma|k}(b),\pi(\sigma)) <
\frac{\epsilon}{2}$. If $k>max(k_{b},{\overline k})$ then, by
equation~\ref{close2}, for any $x\in B$
\[
d_{E}(f_{\sigma|k}(x),\pi(\sigma)) \leq d_{E}(f_{\sigma|k}(x),f_{\sigma|k}(b))
+ d_{E}(f_{\sigma|k}(b),\pi(\sigma)) < \frac{\epsilon}{2}+\frac{\epsilon}%
{2}=\epsilon.
\]

\end{proof}

\begin{theorem}
[A Point-Fibred IFS has an Attractor]\label{2Implies3Theorem} If $\mathcal{F}=
\left(  \mathbb{R}^{m}; f_{1} ,f_{2},...,f_{N}\right)  $ is a point-fibred
affine IFS, then $\mathcal{F}$ has an attractor $A = \pi(\Omega)$, where $\pi:
\Omega\rightarrow\mathbb{R}^{m}$ is the coding map of $\mathcal{F}$.
\end{theorem}

\begin{proof}
It follows directly from the commutative diagram (\ref{commutediagram}) that
$A$ obeys the self-referential equation (\ref{selfrefeq}). We next show that
$A$ satisfies equation (\ref{convergeq}).

Let $\epsilon>0$. We must show that there is an $M$ such that if $k>M$, then
$d_{\mathbb{H}} (\mathcal{F}^{\circ k}(B), \pi(\Omega)) < \epsilon$. It is
sufficient to let $M = \max(M_{1},M_{2})$, where $M_{1}$ and $M_{2}$ are
defined as follows.

First, let $a$ be an arbitrary element of $A$. Then there exists a $\sigma
\in\Omega$ such that $a=\pi(\sigma)$. By Lemma~\ref{uniformLemma} there is an
$M_{1}$ such that if $k>M_{1}$, then $d_{E}(f_{\sigma|k}(b), a) =
d_{E}(f_{\sigma|k}(b), \pi(\sigma)) < \epsilon$, for all $b\in B$. In other
words, $A$ lies in an $\epsilon$-neighborhood of $\mathcal{F}^{\circ k}(B)$.

Second, let $b$ be an arbitrary element of $B$ and $\sigma$ an arbitrary
element of $\Omega$. If $a := \pi(\sigma) \in A$, then there is an $M_{2}$
such that if $k>M_{2}$, then $d_{E}(f_{\sigma|k}(b), a) = d_{E}(f_{\sigma
|k}(b), \pi(\sigma)) < \epsilon$. In other words, $\mathcal{F}^{\circ k}(B)$
lies in an $\epsilon$-neighborhood of $A$.
\end{proof}

\section{An IFS with an Attractor is Topologically Contractive}

The goal of this section is to establish the implication $(3) \Rightarrow(4)$
in Theorem~\ref{maintheorem}. We will show that if an affine IFS has an
attractor as defined in Defintion~\ref{attractorIFSDef} , then it is a
topological contraction. The proof uses notions involving convex bodies.

\begin{definition}
\label{centrallySymmetricDef} A convex body $K$ is \textit{centrally
symmetric} if it has the property that whenever $x \in K,$ then $-x \in K$.
\end{definition}

A well-known general technique for creating centrally symmetric convex bodies
from a given convex body is provided by the next proposition.

\begin{proposition}
\label{centrallySymmetricProp} If a set $K$ is a convex body in $\mathbb{R}%
^{m},$ then the set $K^{\prime}= K - K$ is a centrally symmetric convex body
in $\mathbb{R}^{m}$.
\end{proposition}

\begin{definition}
[Minkowski Norm]\label{minkowskiNormDef} If $K$ is a centrally symmetric
convex body in $\mathbb{R}^{m},$ then the \textit{Minkowski norm} on
$\mathbb{R}^{m}$ is defined by
\[
\left\Vert x\right\Vert _{K} = \inf\, \{ \lambda\geq0 \, : \, x \in\lambda\,
K\} .
\]

\end{definition}

The next proposition is also well-known.

\begin{proposition}
\label{minkowskiNormProp} If $K$ is a centrally symmetric convex body in
$\mathbb{R}^{m},$ then the function $\left\Vert x \right\Vert _{K}$ defines a
norm on $\mathbb{R}^{m}.$ Moreover, the set $K$ is the unit ball with respect
to the Minkowski norm $\left\Vert x \right\Vert _{K}$.
\end{proposition}

\begin{definition}
[Minkowski Metric]\label{minkowskiMetricDef} If $K$ is a centrally symmetric
convex body in $\mathbb{R}^{m}$ and $\left\Vert x\right\Vert _{K}$ is the
associated Minkowski norm, then define the \textit{Minkowski metric} on
$\mathbb{R}^{m}$ by the rule
\[
d_{K}(x,y) := \left\Vert x-y\right\Vert _{K}.
\]

\end{definition}

While R. Rockafeller \cite{rockafeller} refers to such a metric as a
\textit{Minkowski metric}, the reader should be aware that this term is also
associated with the metric on space-time in theory of relativity. Since, for
any convex body $K$ there are positive numbers $r$ and $R$ such that $K$
contains a ball of radius $r$ and is contained in a ball of radius $R$, the
following proposition is clear.

\begin{proposition}
\label{lipschitzProp} If $d$ is a Minkowski metric, then $d$ is Lipschitz
equivalent to the standard metric $d_{E}$ on $\mathbb{R}^{m}$.
\end{proposition}

\begin{proposition}
\label{transMinkowskiProp} A metric $d \, : \, \mathbb{R}^{m} \times
\mathbb{R}^{m} \rightarrow[0,\infty)$ is a Minkowski metric if and only if it
is translation invariant and distances behave linearly along line segments.
More specifically,
\begin{equation}
d(x+z,y+z)=d(x,y)\qquad\hbox{and} \qquad d(x,(1-\lambda)x+\lambda y)=\lambda
d(x,y) \label{minkowski}%
\end{equation}
for all $x, y, z \in\mathbb{R}^{m}$ and all $\lambda\in\mathbb{[}0,1]$.
\end{proposition}

\begin{proof}
For a proof see Rockafeller \cite{rockafeller} pp.131-132.
\end{proof}

\begin{definition}
[Topologically Contractive IFS]\label{topologicalContractiveIFSDef} An IFS
$\mathcal{F}=\{\mathbb{R}^{m};f_{1},f_{2},...,f_{N}\}$ is called
\textit{topologically contractive} if there is a convex body $K$ such that
$\mathcal{F}\left(  K\right)  \subset int(K)$.
\end{definition}

The proof of Theorem \ref{3Implies4Theorem} relies on the following lemma
which is easily proved.

\begin{lemma}
\label{lin-conv} If $g \,:\, \mathbb{R}^{m}\rightarrow\mathbb{R}^{m}$ is
affine and $S\subset\mathbb{R}^{m},$ then $g(conv(S))=conv(g(S))$.
\end{lemma}

\begin{theorem}
[The Existence of an Attractor Implies a Topological Contraction]%
\label{3Implies4Theorem} For an affine IFS $\mathcal{F}=\{\mathbb{R}^{m}%
;f_{1},f_{2},...,f_{N}\}$, if there exists an attractor $A \in\mathbb{H}$ of
the affine IFS $\mathcal{F}=\{\mathbb{R}^{m};f_{1},f_{2},...,f_{N}\}$, then
$\mathcal{F}$ is topologically contractive.
\end{theorem}

\begin{proof}
The proof of this theorem unfolds in three steps.

\begin{enumerate}
\item There exists a convex body $K_{1}$ and a positive integer $t$ with the
property that $\mathcal{F}^{\circ t}\left(  K_{1}\right)  \subset int \left(
K_{1}\right)  $.

\item The set $K_{1}$ is used to define a convex body $K_{2}$ such that $L_{n}
\left(  K_{2}\right)  \subset int \left(  K_{2}\right)  ,$ where $f_{n}(x) =
L_{n} x + a_{n}$ and $n = 1,2 , \dots,N$.

\item There is a positive constant $c$ such that the set $K = cK_{2}$ has the
property $\mathcal{F} \left(  K \right)  \subset int \left(  K \right)  $.
\end{enumerate}

\noindent Proof of Step (1). Let $A$ denote the attractor of\ $\mathcal{F}$.
Let $A_{\rho}=\{x\in\mathbb{R}^{m}:d_{\mathbb{H}}(\left\{  x\right\}
,A)\leq\rho\}$ denote the dilation of $A$ by radius $\rho>0$. Since we are
assuming $\lim_{k\rightarrow\infty}d_{\mathbb{H}}(\mathcal{F}^{\circ
k}(A_{\rho}),A) = 0, $ we can find an integer $t$ so that $d_{\mathbb{H}
}(\mathcal{F}^{\circ t}(A_{1}),A) < 1$. Thus,
\begin{equation}
\label{custom}\mathcal{F}^{\circ t}(A_{1})\subseteq int(A_{1}).
\end{equation}
\noindent If we let $K_{1} := conv\left(  A_{1}\right)  $, then
\begin{align*}
\mathcal{F}^{\circ t}\left(  K_{1}\right)   &  = \bigcup_{i_{1} \in\Omega}\;
\bigcup_{i_{2} \in\Omega} \cdots\bigcup_{i_{t} \in\Omega} (f_{i_{1}}\circ
f_{i_{2}} \circ\dots\circ f_{i_{t}}) \left(  conv(A_{1}) \right) \\
&  = \bigcup_{i_{1} \in\Omega}\; \bigcup_{i_{2} \in\Omega} \cdots
\bigcup_{i_{t} \in\Omega} conv\left(  f_{i_{1}}\circ f_{i_{2}} \circ\dots\circ
f_{i_{t}} ( A_{1} ) \right)  \quad\quad(\text{by Lemma~\ref{lin-conv}})\\
&  \subseteq\bigcup_{i_{1} \in\Omega}\; \bigcup_{i_{2} \in\Omega}
\cdots\bigcup_{i_{t} \in\Omega} conv\left(  int(A_{1}) \right)  = conv(
int(A_{1})) \quad(\text{by inclusion (5.2))}\\
&  \subseteq int(conv(A_{1})) = int\left(  K_{1}\right)  .
\end{align*}
This argument completes the proof of Step (1). \vskip 1mm \noindent Proof of
Step (2). Consider the set
\[
K_{2}:= \sum_{k=0}^{t-1} (conv(\mathcal{F}^{\circ k}(K _{1}) -
conv(\mathcal{F}^{\circ k}(K_{1}))).
\]

\noindent The set $K_{2}$ is a centrally symmetric convex body because it is a
finite Minkowski sum of centrally symmetric convex bodies. If any affine map
$f_{n}$ in $\mathcal{F}$ is written $f_{n}(x)=L_{n} x+a_{n},$ where
$L_{n}:\mathbb{R}^{m}\rightarrow\mathbb{R}^{m}$ denotes the linear part, then
\begin{align*}
L_{n}(K_{2})  &  = \sum_{k=0}^{t-1}L_{n} \Bigl( conv(\mathcal{F}^{\circ
k}(K_{1}) - conv(\mathcal{F}^{\circ k}(K_{1})) \Bigr)\quad\text{ (since }
L_{n}\text{ is a linear map)}\\
&  = \sum_{k=0}^{t-1} \Bigl( conv(L_{n}\left(  \mathcal{F}^{\circ k}\left(
K_{1}\right)  \right)  ) - conv(L_{n}\left(  \mathcal{F}^{\circ k}%
(K_{1})\right)  )\Bigr) \quad\text{(by Lemma~\ref{lin-conv})}\\
&  = \sum_{k=0}^{t-1}\Bigl( conv(f_{n}\left(  \mathcal{F}^{\circ k}\left(
K_{1}\right)  \right)  ) - conv(f_{n}\left(  \mathcal{F}^{\circ k}%
(K_{1})\right)  )\Bigr) \quad\text{ (since the }a_{n}\text{s cancel)}\\
&  \subseteq\sum_{k=0}^{t-1} \left(  conv(\mathcal{F}^{\circ\left(
k+1\right)  }\left(  K_{1}\right)  ) - conv(\mathcal{F}^{\circ\left(
k+1\right)  }(K_{1}))\right) \\
&  = \Bigl( conv(\mathcal{F}^{\circ t}(K_{1}) - conv(\mathcal{F}^{\circ
t}(K_{1})) \Bigr) + \sum_{k=1}^{t-1} \Bigl( conv(\mathcal{F}^{\circ k}(K_{1})
- conv(\mathcal{F}^{\circ k}(K_{1})) \Bigr)\\
&  \subseteq\left(  int\left(  K_{1}\right)  - int\left(  K_{1}\right)
\right)  + \sum_{k=1}^{t-1}(conv(\mathcal{F}^{\circ k}(K_{1})-
conv(\mathcal{F}^{\circ k}(K_{1})))\quad\text{ (by Step 1)}\\
&  = int(K_{2}).
\end{align*}

The second to last inclusion follows from the fact that \linebreak%
$f_{n}\left(  \mathcal{F}^{\circ k}\left(  K_{1}\right)  \right)
\subset\mathcal{F}^{\circ\left(  k+1\right)  }\left(  K_{1}\right)  $. The
last equality follows from the fact that if $\mathcal{O}$ and $\mathcal{C}$
are symmetric convex bodies in $\mathbb{R}^{m},$ then $int(\mathcal{O}%
)+\mathcal{C} = int\left(  \mathcal{O}+\mathcal{C}\right)  $. We have now
completed the proof of Step (2). \vskip 1mm

\noindent Proof of Step (3). It follows from Step (2) and the compactness of
$K_{2}$ that there is a constant $\alpha\in(0,1)$ such that $d_{K _{2}}%
(L_{n}(x),L_{n}(y)) < \alpha d_{K_{2}}(x,y)$ for all $x,y\in\mathbb{R}^{m}$
and all $n = 1,2, \dots, N$.

Let
\[
c > \frac{r}{\left(  1 - \alpha\right)  },
\]
where $r =\max\{ d_{K _{2}} ( a_{1},0), d_{K _{2}}( a_{2},0), \dots, d_{K
_{2}} ( a_{N},0) \}$. If $x \in cK_{2}$ and $f(x) = Lx +a$ is any function in
the IFS $\mathcal{F}$, then
\begin{align*}
\left\Vert f(x) \right\Vert _{K_{2}}  &  = d_{K_{2}}\left(  f\left(  x\right)
,0\right)  = d_{K_{2}}\left(  Lx+a,0\right)  \leq d_{K_{2}}\left(  Lx+a, Lx
\right)  + d_{K_{2}}\left(  Lx,0\right) \\
&  = d_{K_{2}}\left(  a,0\right)  + d_{K_{2}}\left(  Lx,0\right)
\ \ (\text{by Equation (\ref{minkowski})} )\\
&  < r + \alpha d_{K_{2}}\left(  x,0\right)  = r + \alpha\, \left\Vert x
\right\Vert _{K_{2}}\\
&  \leq r + \alpha c < (c - \alpha c) + \alpha c = c.
\end{align*}
This inequality shows that $\mathcal{F}\left(  cK_{2}\right)  \subset
int(cK_{2})$.
\end{proof}

\section{A Non-Antipodal Affine IFS is Hyperbolic}

Let $S^{m-1}\subset\mathbb{R}^{m}$ denote the unit sphere in $\mathbb{R}^{m}$.
For a convex body $K \subset\mathbb{R}^{m}$ and $u\in S^{m-1}$ there exists a
pair, $\left\{  H_{u},H_{-u}\right\}  ,$ of distinct supporting hyperplanes of
$K$, each orthogonal to $u$ and with the property that they both intersect
$\partial K$ but contain no points of the interior of $K$. Since by definition
a convex body has non-empty interior, this pair will be unique. The pair
$\left\{  H_{u},H_{-u}\right\}  $ is usually referred to as the two
\textit{supporting hyperplanes} of $K$ orthogonal to $u$. (See Moszynska
\cite{moszynska}, p.14.)

\begin{definition}
[Antipodal Pairs]\label{antipodalPointsDef} If $K \subset\mathbb{R}^{m}$ is a
convex body and $u \in S^{m-1},$ then define
\[
\begin{aligned}
\mathcal{A}_{u}&:=\mathcal{A}_{u}(K) = \left\{ \left(
p,q\right) \in\left( H _{u}\cap \partial K\right)
\times\left( H_{-u} \cap \partial K\right) \right\} \quad
\text{ and } \\
\mathcal{A} &:= \mathcal{A}(K)
=\bigcup_{u \in S^{m-1}} \mathcal{A}_{u}\text{.} \end{aligned}
\]
We say that $\left(  p,q\right)  $ is an \textit{antipodal pair} of points
with respect to $K$ if $\left(  p,q\right)  \in\mathcal{A}$.
\end{definition}

\begin{definition}
[Diametric Pairs]\label{diametricPointsDef} If $K \subset\mathbb{R}^{m}$ is a
convex body, and $u\mathbf{\in}{S^{m-1}},$ then define the \textit{diameter}
of $K$ \textit{in the direction }$u$ to be
\[
D(u)=\max\{\left\Vert x-y\right\Vert _{2} :x,y\mathbf{\in} K, x-y=\alpha
u,\alpha\in\mathbb{R}\}.
\]
The maximum is achieved at some pair of points belonging to $\partial K$
because $K\times K$ is convex and compact, and $\left\Vert x-y\right\Vert _{2}
$ is continuous for $(x,y)\in K \times K$. Now define
\[
\begin{aligned}
\mathcal{D}_{u}&=\{(p,q)\in\partial K\times\partial K
:D(u)=\left\Vert q-p\right\Vert_2 \} \quad \text{ and } \\
\mathcal{D}
&=\bigcup_{u \in S^{m-1}}\mathcal{D}_{u}\text{.} \end{aligned}
\]
We say that $\left(  p,q\right)  \in\mathcal{D}_{u}$ is a \textit{diametric
pair} of points in the direction of $u,$ and that $\mathcal{D}$ is the set of
diametric pairs of points of $K.$
\end{definition}

\begin{definition}
[Strictly Convex]A convex body $K$ is \textit{strictly convex} if, for every
two points$\ x,y\in K$, the open line segment joining $x$ and $y$ is contained
in the interior of $K$.
\end{definition}

We write $xy$ to denote the closed line segment with endpoints at $x$ and $y$
so that $y-x$ is the vector, in the direction from $x$ to $y,$ whose magnitude
is the length of $xy$.

\begin{theorem}
\label{keyADTheorem} If $K \subset\mathbb{R}^{m}$ is a convex body, then the
set of antipodal pairs of points of $K$ is the same as the set of diametric
pairs of points of $K,$ i.e.,
\[
\mathcal{A=D}.
\]

\end{theorem}

\begin{proof}
First we show that $\mathcal{A}\subseteq\mathcal{D}$. If $(p,q)\in
\mathcal{A},$ then $p\in H_{u} \cap\partial K$ and $q\in H_{-u} \cap\partial
K$ for some $u\in S^{m-1}$. Clearly any chord of $K$ parallel to $pq$ lies
entirely in the region between $H_{u}$ and $H_{-u}$ and therefore cannot have
length greater than that of $pq$. So $D(q-p)=\left\Vert q-p\right\Vert $ and
$\left(  p,q\right)  \in\mathcal{D}_{q-p}\subseteq\mathcal{D}$. Note, for use
later in the proof, that if $K$ is strictly convex, then $pq$ is the unique
chord of maximum length in its direction.

Conversely, to show that $\mathcal{D}\subseteq\mathcal{A}$, first consider the
case where $K$ is a strictly convex body. For each $u\in S^{m-1},$ consider
the points $x_{u}\in H_{u}\cap\partial K$ and $x_{-u}\in H_{-u}\cap\partial
K$. The continuous function $f:S^{m-1}\rightarrow S^{m-1}$ defined by
\[
f(u)=\frac{x_{u}-x_{-u}}{\left\Vert x_{u}-x_{-u}\right\Vert _{2} }
\]
has the property that $\langle f(u),u\rangle>0$ for all $u$. In other words,
the angle between $u$ and $f(u)$ is less than $\frac{\pi}{2}$. But it is an
elementary exercise in topology (see, for example, Munkres \cite{munkres},
problem 10, page 367) that if $f:S^{m-1}\rightarrow S^{m-1}$ maps no point $x$
to its antipode $-x$, then $f$ has degree $1$ and, in particular, is
surjective. To show that $\mathcal{D}\subseteq\mathcal{A}$, let $(p,q)\in
\mathcal{D}_{v}$ for some $v\in S^{m-1}$. By the surjectivity of $f$ there is
$u\in S^{m-1}$ such that $f(u)=v$. According to the last sentence of the
previous paragraph, $x_{u}x_{-u}$ is the unique longest chord parallel to $v$.
Therefore $p=x_{u}$ and $q=x_{-u}$ and consequently $(p,q)\in\mathcal{A}_{u}$.

The case where $K$ is not strictly convex is treated by a standard limiting
argument. Given a vector $v\in S^{m-1}$ and a longest chord $pq$ parallel to
$v$, we must prove that $(p,q)\in\mathcal{A}$. Since $K$ is the intersection
of all strictly convex bodies containing $K$, there is a sequence $\left\{
K_{k}\right\}  $ of strictly convex bodies containing $K$ with the following
two properties.\medskip

1. There is a longest chord $p_{k}q_{k}$ of $K_{k}$ parallel to $u$ such that
$\lim_{k\rightarrow\infty}\left\Vert p_{k}-q_{k}\right\Vert _{2}=\left\Vert
p-q\right\Vert _{2},$ and the limits $\lim_{k\rightarrow\infty}p_{k}=\tilde
{p}\in K$ and $\lim_{k\rightarrow\infty}q_{k}=\tilde{q}\in K$ exist.\medskip

By the result \ for the strictly convex case, there is a sequence of vectors
$u_{k}\in S^{m-1}$ such that $p_{k}= K_{k}\cap H_{u_{k}}(K_{k})$ and
$q_{k}=K_{k}\cap H_{-u_{k}}(K_{k})$. By perhaps going to a subsequence\medskip

2. $\lim_{k\rightarrow\infty}u_{k}=u\in S^{m-1}$ exists.\medskip

It follows from item 1 that $\left\Vert \tilde{p}-\tilde{q}\right\Vert _{2}
=\left\Vert p-q\right\Vert _{2} $ and $\tilde{p}-\tilde{q}$ is parallel to
$v$. Therefore, $\tilde{p}\tilde{q}$ as well as $pq$, are longest chords of
$K$ parallel to $v$. It follows from 2 that if $H$ and $H^{\prime}$ are the
hyperplanes orthogonal to $u$ through $\tilde{p}$ and $\tilde{q}$
respectively, then $H$ and $H^{\prime}$ are parallel supporting hyperplanes of
$K$. Therefore, necessarily $p\in H$ and $q\in H^{\prime}$, and consequently
$(p,q)\in\mathcal{A}_{u}\subset\mathcal{A}$.
\end{proof}

\begin{definition}
[Non-Antipodal IFS]\label{nonantipodalIFSDef} If $K \subset\mathbb{R}^{m}$ is
a convex body, then $f:\mathbb{R}^{m}\rightarrow\mathbb{R}^{m}$ is
\textit{non-antipodal} with respect to $K$ if $f(K )\subseteq K$ and $\left(
x,y\right)  \in\mathcal{A}\left(  K\right)  $ implies $\left(
f(x),f(y)\right)  \notin\mathcal{A} \left(  K\right)  $. If $\mathcal{F} =
\{\mathbb{R}^{m}; f_{1},f_{2},...,f_{N}\}$ is an iterated function system with
the property that each $f_{n}$ is non-antipodal with respect to $K,$ then
$\mathcal{F}$ is called \textit{non-antipodal} with respect to $K$.
\end{definition}

The next proposition gives the implication $(4) \Rightarrow(5)$ in
Theorem~\ref{maintheorem}. The proof is clear.

\begin{proposition}
[A Topological Contraction is Non-Antipodal]%
\label{topContractImpliesnonAntpodalProp} If \linebreak$\mathcal{F} =
\{\mathbb{R}^{m}; f_{1},f_{2},...,f_{N}\}$ is an affine iterated function
system with the property that there exists a convex body $K\subset
\mathbb{R}^{m}$ such that $f_{n}(K)\subset int(K) $ for all $n = 1,2 ,
\dots,n,$ then $\mathcal{F} $ is non-antipodal with respect to $K$.
\end{proposition}

The next theorem provides the implication that $(5) \Rightarrow(1)$ in
Theorem~\ref{maintheorem}.

\begin{theorem}
\label{5Implies1Theorem} If the affine IFS $\mathcal{F} = \left(
\mathbb{R}^{m}; f_{1} ,f_{2},...,f_{N}\right)  $ is non-antipodal with respect
to a convex body $K$, then $\mathcal{F}$ is hyperbolic.
\end{theorem}

\begin{proof}
Assume that $K$ is a convex body such that $f$ is non-antipodal with respect
to $K$ for all $f \in\mathcal{F}$. Let $C = K-K$ and let $f(x) = Lx + a
\in\mathcal{F}$, where $L$ is the linear part of $f$. By
Proposition~\ref{centrallySymmetricProp}, the set $C$ is a centrally symmetric
convex body and
\[
L(C) = L(K)-L(K) = f(K) - f(K)\subseteq K-K = C.
\]
\vskip 2mm

We claim that $L(C) \subset int(C)$. Since $C$ is compact and $L$ is linear,
to prove the claim it is sufficient to show that $L(x) \notin\partial C$ for
all $x \in\partial C$. By way of contradiction, assume that $x\in\partial C$
and $L(x) \in\partial C$. Then the vector $x$ is a longest vector in $C$ in
its direction. Since $x\in C= K-K$ there are $x_{1}, x_{2} \in\partial K$ such
that $x = x_{1}-x_{2}$, and $(x_{1},x_{2}) \in{\mathcal{D}}(K) = {\mathcal{A}%
}(K)$, where the last equality is by Theorem~\ref{keyADTheorem}. So
$(x_{1},x_{2})$ is an antipodal pair with respect to $K$. Likewise, since $Lx$
is a longest vector in $C$ in its direction, there are $y_{1}, y_{2}
\in\partial K$ such that $Lx = y_{1}-y_{2}$, and $(y_{1},y_{2}) \in
{\mathcal{D}}(K) = {\mathcal{A}}(K)$. Therefore
\[
f(x_{2}) - f(x_{1}) = L(x_{2}) - L(x_{1}) = L(x_{2}-x_{1}) = Lx = y_{1}%
-y_{2},
\]
which implies that $(f_{n}(x_{1}),f_{n}(x_{2})) \in{\mathcal{D}}(K) =
{\mathcal{A}} (K)$, contradicting that $f$ is non-antipodal with respect to
$K$.

If $d_{C}$ denotes the Minkowski metric with respect to the centrally
symmetric convex body $C$, then by Proposition~\ref{minkowskiNormProp} $C$ is
the unit ball centered at the origin with respect to this metric. Since $C$ is
compact, the containment $L(C) \subset int(C)$ implies that there is an
$\alpha\in[0,1)$ such that $\left\Vert Lx \right\Vert _{C} < \alpha\,
\left\Vert x \right\Vert _{C}$ for all $x \in\mathbb{R}^{m}$. Then
\[
\begin{aligned}
d_{C}(f(x),f(y))
&= \left\Vert f(x) - f(y) \right\Vert_{C} =
\left\Vert Lx - Ly \right\Vert_{C} \\
&=
\left\Vert L(x-y)\right\Vert_{C} <
\alpha \, \left\Vert x-y\right\Vert_{C} =
\alpha \, d_{C}(x,y).
\end{aligned}
\]
Therefore $d_{C}$ is a metric for which each function in the IFS is a
contraction. By Proposition~\ref{lipschitzProp}, $d_{C}$ is Lipschitz
equivalent to the standard metric.
\end{proof}

\section{An Answer to the Question of Kameyama}

We now turn to the proof of Theorem~\ref{kameyamaTheorem}, the theorem that
settles the question of Kameyama. If $X \subseteq\mathbb{R}^{m}$ and
$\mathcal{F} = \left(  X; f_{1} ,f_{2},...,f_{N}\right)  $ is an IFS on $X$,
then the definitions of \textit{coding map} and \textit{point-fibred} for
$\mathcal{F}$ are exactly the same as Definitions~\ref{codingMapDef} and
\ref{point-fibredDef}, with $\mathbb{R}^{m}$ replaced by $X$. The proof of
Theorem~\ref{kameyamaTheorem} requires the following proposition.

\begin{proposition}
\label{lem1} If $X \subseteq\mathbb{R}^{m}$ and $\mathcal{F} = \left(  X;
f_{1} ,f_{2},...,f_{N}\right)  $ is an IFS with a coding map $\pi:
\Omega\rightarrow\mathbb{R}^{m}$ such that $\pi(\Omega)=X,$ then $\mathcal{F}$
is point-fibred on $X$.
\end{proposition}

\begin{proof}
By Definition \ref{point-fibredDef}, we must show that $\lim_{k\rightarrow
\infty}f_{\sigma_{1}}\circ f_{\sigma_{2}} \circ\cdots\circ f_{\sigma_{k}}(x) $
exists, is independent of $x \in X$, and is continuous as a function of
$\sigma= \sigma_{1}\sigma_{2} \cdots\in\Omega$. We will actually show that
$\lim_{k\rightarrow\infty}f_{\sigma_{1}}\circ f_{\sigma_{2}} \circ\cdots\circ
f_{\sigma_{k}}(x) = \pi(\sigma) $.

Since $\pi$ is a coding map, we know by Definition \ref{codingMapDef} that
$f_{n} \circ\pi(\sigma) = \pi\circ s_{n} (\sigma),$ for all $n =1, 2, \dots,
N$. By assumption, if $x$ is any point in $X,$ then there is a $\tau\in\Omega$
such that $\pi(\tau) = x$. Thus
\begin{align*}
\lim_{k\rightarrow\infty}f_{\sigma_{1}}\circ f_{\sigma_{2}}\circ\cdots\circ
f_{\sigma_{k}} (x)  &  \mathbb{=}\lim_{k\rightarrow\infty} f_{\sigma_{1}}\circ
f_{\sigma_{2}}\circ\cdots\circ f_{\sigma_{k}}(\pi(\tau)\mathbb{)}\text{ (since
} \pi(\tau) = x)\\
&  = \lim_{k\rightarrow\infty}\pi(s_{\sigma_{1}}\circ s_{\sigma_{2}}
\circ...\circ s_{\sigma_{k}} \circ\tau) \text{ (by Diagram
\ref{commutediagram})}\\
&  = \pi(\lim_{k\rightarrow\infty}s_{\sigma_{1}}\circ s_{\sigma_{2}}
\circ...\circ s_{\sigma_{k}} \circ\tau) \text{ (since $\pi$ is continuous)}\\
&  = \pi(\sigma).
\end{align*}

\end{proof}

\begin{theorem}
\label{kameyamatheorem} If $\mathcal{F}= \left(  \mathbb{R}^{m}; f_{1}
,f_{2},...,f_{N}\right)  $ is an affine IFS with a coding map $\pi:
\Omega\rightarrow X$, then $\mathcal{F}$ is point-fibred when restricted to
the affine hull of $\pi(\Omega)$. In particular, if $\pi(\Omega)$ contains a
non-empty open subset of $\mathbb{R}^{m},$ then $\mathcal{F}$ is point-fibred
on $\mathbb{R}^{m}$.
\end{theorem}

\begin{proof}
Let $A:= \pi(\Omega)$. Since $f_{n}(A) \subseteq A$ for all $n$, the
restriction of the IFS $\mathcal{F}$ to $A$, namely $\mathcal{F}|_{A} := (A ;
f_{1}, f_{2}, \dots, f_{N})$, is well defined. It follows from
Proposition~\ref{lem1} that $\mathcal{F}|_{A}$ is point-fibred and, because
the coding map for a point-fibred IFS is unique,
\[
\pi(\sigma) = \lim_{k\rightarrow\infty}f_{\sigma_{1}}\circ f_{\sigma_{2}}%
\circ\cdots\circ f_{\sigma_{k}}(a)\label{domeq}%
\]
for $\left(  \sigma,a\right)  \in\Omega\times A$. It only remains to show that
the restriction $\mathcal{F}|_{\text{aff}( A )} := (\text{aff}( A ) ; f_{1},
f_{2}, \dots, f_{N})$ of the affine IFS $\mathcal{F}$ to the affine hull of
$A$ is point-fibred.

Let $x \in\text{aff}( A )$, the affine hull of $A$. It is well known that any
point in the affine hull can be expressed as a sum, $x=\sum_{p=0}^{m}%
\lambda_{p}a_{p}$ for some $\lambda_{0},\lambda_{1},...,\lambda_{m}%
\in\mathbb{R}$ such that $\sum_{p=0}^{m}\lambda_{p}=1$ and $a_{0}%
,a_{1},...,a_{m}\in A$. Hence, for $\left(  \sigma,x\right)  \in\Omega
\times\text{aff}( A)$,
\begin{align*}
\lim_{k\rightarrow\infty}f_{\sigma_{1}}\circ f_{\sigma_{2}}\circ\cdots\circ
f_{\sigma_{k}}(x)  &  =\lim_{k\rightarrow\infty}f_{\sigma_{1}}\circ
f_{\sigma_{2}}\circ\cdots\circ f_{\sigma_{k}}(\sum_{p=0}^{m}\lambda_{p}%
a_{p})\text{, }\\
&  =\lim_{k\rightarrow\infty}\sum_{p=0}^{m}\lambda_{p}f_{\sigma_{1}}\circ
f_{\sigma_{2}}\circ\cdots\circ f_{\sigma_{k}}(a_{p})\\
&  =\sum_{p=0}^{m}\lambda_{p}\pi(\sigma)=\pi(\sigma).
\end{align*}

\end{proof}

Theorem~\ref{kameyamaTheorem} now follows easily from
Theorem~\ref{kameyamatheorem} and Theorem~\ref{maintheorem}.

\begin{proof}
(of Theorem~\ref{kameyamaTheorem}) Let $A:= \pi(\Omega)$ and let $\dim\,
\text{aff}( A ) = k \leq m$. It is easy to check from the commuting
diagram~\ref{commutediagram} that $f(A) \subseteq A$ for each $f\in
\mathcal{F}$ implies that $f( \text{aff}( A )) \subseteq\text{aff}( A )$ for
each $f\in\mathcal{F}$. Since $\text{aff}( A )$ is isomorphic to
$\mathbb{R}^{k}$, Theorem~\ref{maintheorem} can be applied to the IFS
$\mathcal{F}|_{\text{aff}( A )} := (\text{aff}( A ); f_{1},f_{2},...f_{N} )$
to conclude that, since it is point-fibred, $\mathcal{F}|_{\text{aff}( A )}$
is also hyperbolic.
\end{proof}

Note that the IFS $(\mathbb{R};f)$, where $f(x)=2x+1$, is not hyperbolic on
$\mathbb{R}$, but it is hyperbolic on the affine subspace $\{-1\}\subset
\mathbb{R}$.

\section{Concluding Remarks}

Recently it has come to our attention that another condition, equivalent to
conditions $(1)-(5)$ in our main result, Theorem 1.1, is $(6)$ $\mathcal{F}%
$\textit{ has joint spectral radius less than one}. (We define the joint
spectral radius (JSR) of an affine IFS to be the joint spectral radius of the
set of linear factors of its maps.) This information is important because it
connects our approach to the rapidly growing literature about JSR, see for
example \cite{berger}, \cite{daubechies}, and works that refer to these.

Since Example 3.3 and the results presented by Blondel, Theys, and Vladimirov
\cite{theys} indicate there is no general fast algorithm which will determine
whether or not the joint spectral radius of an IFS\ is less than one, we feel
that Theorem 1.1 is important because it provides an easily testable condition
that an IFS\ has a unique attractor. In particular, the topologically
contractive and non-antipodal conditions (conditions 4 and 5) provide
geometric/visual tests, which can easily be checked for any affine IFS. In
addition to yielding the existence of an attractor, these two conditions also
provide information concerning the location of the attractor. (For example,
the attractor is a subset of a particular convex body.) We also anticipate
that Theorem 1.1 can be generalized into other broader classes of functions,
where the techniques developed for the theory of joint spectal radius will not apply.


\end{document}